\documentclass{lms}

\usepackage{amsmath}
\usepackage{amssymb}
\usepackage[english]{babel}
\usepackage[all]{xy}

\newcommand{\zb}{\mathbb{Z}}

\newcommand{\qb}{\mathbb{Q}}
\newcommand{\cb}{\mathbb{C}}

\newcommand{\pb}{\mathbb{P}}

\newcommand{\al}{\alpha}
\newcommand{\eps}{\varepsilon}
\newcommand{\ld}{\lambda}
\newcommand{\vph}{\varphi}

\newcommand{\fl}{\rightarrow}
\newcommand{\gfl}{\longrightarrow}
\newcommand{\infl}{\rightarrowtail}

\newcommand{\dr}{\ar@{->}[r]}
\newcommand{\dri}{\ar@{>->}[r]}
\newcommand{\drp}{\ar@{-->}[r]}
\newcommand{\dre}{\ar@{->>}[r]}
\newcommand{\dreg}{\ar@{=}[r]}
\newcommand{\drm}{\ar@{^{(}->}[r]}
\newcommand{\ddr}{\ar@{->}[rr]}
\newcommand{\ddre}{\ar@{->>}[rr]}
\newcommand{\ddreg}{\ar@{=}[rr]}
\newcommand{\ha}{\ar@{->}[u]}
\newcommand{\hae}{\ar@{->>}[u]}
\newcommand{\hap}{\ar@{-->}[u]}
\newcommand{\ham}{\ar@{^{(}->}[u]}
\newcommand{\hham}{\ar@{^{(}->}[uu]}
\newcommand{\hag}{\ar@{->}[ul]}
\newcommand{\hagm}{\ar@{^{(}->}[ul]}
\newcommand{\hagp}{\ar@{-->}[ul]}
\newcommand{\hdr}{\ar@{->}[ur]}
\newcommand{\hdrm}{\ar@{^{(}->}[ur]}
\newcommand{\hdri}{\ar@{>->}[ur]}
\newcommand{\hdre}{\ar@{->>}[ur]}
\newcommand{\hdrp}{\ar@{-->}[ur]}
\newcommand{\bas}{\ar@{->}[d]}
\newcommand{\bbas}{\ar@{->}[dd]}
\newcommand{\basm}{\ar@{^{(}->}[d]}
\newcommand{\basp}{\ar@{-->}[d]}
\newcommand{\baseg}{\ar@{=}[d]}
\newcommand{\bbaseg}{\ar@{=}[dd]}
\newcommand{\bdr}{\ar@{->}[dr]}
\newcommand{\bdre}{\ar@{->>}[dr]}
\newcommand{\bbdr}{\ar@{->}[ddr]}
\newcommand{\bddr}{\ar@{->}[drr]}
\newcommand{\bg}{\ar@{->}[dl]}
\newcommand{\bgm}{\ar@{^{(}->}[dl]}
\newcommand{\bgp}{\ar@{-->}[dl]}
\newcommand{\bggp}{\ar@{-->}[dll]}
\newcommand{\bbgp}{\ar@{-->}[ddl]}

\newcommand{\cat}{\mathcal{C}}
\newcommand{\dc}{\mathcal{D}}
\newcommand{\ec}{\mathcal{E}}
\newcommand{\hc}{\mathcal{H}}
\newcommand{\lc}{\mathcal{L}}

\newcommand{\pc}{\mathcal{P}}
\newcommand{\rc}{\mathcal{R}}

\newcommand{\wc}{\mathcal{W}}

\newcommand{\proj}{\operatorname{proj}}
\newcommand{\modb}{\mathrm{mod}B}
\newcommand{\moda}{\operatorname{mod}A}

\newcommand{\ind}{\operatorname{ind}}
\newcommand{\coi}{\operatorname{coind}}
\newcommand{\mte}{\operatorname{mt}(\eps)}

\newcommand{\susp}{\Sigma}
\newcommand{\susm}{\Sigma^{-1}}
\newcommand{\ke}{\operatorname{Ker}}
\newcommand{\im}{\operatorname{Im}}
\newcommand{\coke}{\operatorname{Coker}}
\newcommand{\homb}{\operatorname{Hom}_B}
\newcommand{\ext}{\mathrm{Ext}}
\newcommand{\homph}{\operatorname{Hom}}

\newcommand{\kzero}{K_{0}}

\newcommand{\ko}{K_0(\operatorname{mod}B)}

\newcommand{\kop}{K_0(\operatorname{proj}B)}

\newcommand{\clm}{\cat(L,\susp M)}
\newcommand{\elm}{\ext^1_\ec(L,M)}
\newcommand{\clmy}{\cat(L,\susp M)_{\class}}
\newcommand{\elmy}{\cat(L,\susp M)_{\dclass}}

\newcommand{\pclm}{\pb\cat(L,\susp M)}
\newcommand{\pelm}{\pb\elm}
\newcommand{\pcml}{\pb\cat(M,\susp L)}
\newcommand{\peml}{\pb\ext^1_\ec(M,L)}
\newcommand{\pclmy}{\pb\cat(L,\susp M)_{\class}}
\newcommand{\pelmy}{\pb\cat(L,\susp M)_{\dclass}}
\newcommand{\pcmly}{\pb\cat(M,\susp L)_{\class}}
\newcommand{\pemly}{\pb\cat(M,\susp L)_{\dclass}}
\newcommand{\class}{\langle \, Y \rangle}
\newcommand{\dclass}{{\langle\langle \, Y \rangle\rangle}}

\newcommand{\wlmg}{W_{L M}^Y(g)}
\newcommand{\ewlmg}{\wc_{L M}^Y(g)}
\newcommand{\wmlg}{W_{M L}^Y(g)}
\newcommand{\ewmlg}{\wc_{M L}^Y(g)}

\newcommand{\wef}{W_{L M}^Y(e,f)}
\newcommand{\ewef}{\wc_{L M}^Y(e,f)}
\newcommand{\wefg}{W_{L M}^Y(e,f,g)}
\newcommand{\ewefg}{\wc_{L M}^Y(e,f,g)}
\newcommand{\elef}{L(e,f)}
\newcommand{\eelef}{\lc(e,f)}
\newcommand{\elun}{L_1(e,f)}
\newcommand{\eelun}{\lc_1(e,f)}
\newcommand{\elde}{L_2(e,f)}
\newcommand{\eelde}{\lc_2(e,f)}
\newcommand{\cef}{C_{e,f}}
\newcommand{\ecef}{\cat_{e,f}}
\newcommand{\cefg}{C_{e,f}(Y,g)}
\newcommand{\ecefg}{\cat_{e,f}(Y,g)}
\newcommand{\vg}{V_{L M}^Y(g)}

\newcommand{\wmlefg}{W_{M L}^Y(f,e,g)}
\newcommand{\ewmlefg}{\wc_{M L}^Y(f,e,g)}
\newcommand{\wg}{W_{L M}^Y(g)}
\newcommand{\ewg}{\wc_{L M}^Y(g)}

\newcommand{\gre}{\operatorname{Gr}_e}
\newcommand{\grf}{\operatorname{Gr}_f}
\newcommand{\grg}{\operatorname{Gr}_g}
\newcommand{\gr}{\operatorname{Gr}}

\newcommand{\vdim}{\underline{\dim}}
\newcommand{\rep}{\operatorname{rep}}
\newcommand{\yc}{\mathcal{Y}}
\newcommand{\End}{\operatorname{End}}
\newcommand{\ens}[1]{\left\{ #1 \right\}}
\newcommand{\ps}[2]{\langle #1 \,,\,#2\rangle}
\newcommand{\gruv}{\operatorname{Gr}_{U,V}}
\newcommand{\had}{\ar@[]}

\newcommand{\per}{\operatorname{per}}
\newcommand{\perg}{\per\Gamma}
\newcommand{\fc}{\mathcal{F}}

\newcommand{\is}{\underline{i}}
\newcommand{\js}{\underline{j}}
\newcommand{\ls}{\underline{l}}
\newcommand{\xs}{\underline{x}}

\newtheorem{theo}{Theorem}[section] 
\newtheorem{lem}[theo]{Lemma}     
\newtheorem{cor}[theo]{Corollary}
\newtheorem{prop}[theo]{Proposition}

\newnumbered{defi}{Definition}
\newnumbered{remark}{Remark}

\title[Cluster characters II]
 {Cluster characters II: A multiplication formula} 

\author{Yann Palu}

\classno{18E30 (primary), 16G20, 13F60, 14A10 (secondary)}

\begin{document}
\maketitle

\begin{abstract}
Let $\mathcal{C}$ be a Hom-finite triangulated
2-Calabi--Yau category with a cluster tilting
object. Under some constructibility assumptions
on $\mathcal{C}$ which are satisfied for instance
by cluster categories, by generalized cluster categories
and by stable categories of modules over a preprojective
algebra of Dynkin type, we prove a multiplication formula
for the cluster character associated with any cluster
tilting object. This formula generalizes those
obtained by Caldero--Keller for representation finite
path algebras and by Xiao--Xu for finite-dimensional
path algebras. We prove an analogous formula
for the cluster character defined by Fu--Keller
in the setup of Frobenius categories.
It is similar to a formula obtained
by Geiss--Leclerc--Schr\"oer in the context of
preprojective algebras.
\end{abstract}

\section*{Introduction} 
\label{intro}

In recent years, the link between Fomin--Zelevinsky's
cluster algebras~\cite{FZ1} and the representation theory
of quivers and finite-dimensional algebras has been
investigated intensely, cf. for example the surveys~\cite{BM},
\cite{GLSicra}, \cite{Kcategorification}.
In its most tangible form, this link is given by a map taking
objects of cluster categories to elements of cluster algebras.
Such a map was first constructed by P.~Caldero and
F.~Chapoton~\cite{CC} for cluster categories
and cluster algebras associated with Dynkin quivers.
Another approach, leading to proofs of several conjectures
on cluster algebras in a more general context,
can be found in~\cite{DWZ1},~\cite{DWZ2} (for proofs
relying on the use of a Caldero--Chapoton map, see
\cite{PlamondonHomInfinite}, \cite{PlamondonClusterAlgebras}).

The results of P. Caldero and B.~Keller \cite{CK1} yield two
multiplication formulae for the Caldero--Chapoton map
of cluster categories associated with Dynkin quivers.
The first one categorifies the exchange relations of cluster
variables and only applies to objects $L$ and $M$
such that $\ext^1(L,M)$ is of dimension $1$.
The second one generalizes it to arbitrary dimensions,
and yields some new relations
in the associated cluster algebras.
These relations very much resemble relations in
dual Ringel--Hall algebras~\cite[section 5.5]{SchiffmannHall}.
Motivated by these results,
C. Geiss, B. Leclerc and J. Schr\"oer~\cite{GLSmf}
proved two analogous formulae for module
categories over preprojective algebras.
In this latter situation, the number of
isomorphism classes of indecomposable objects is usually infinite.
Generalizations of the first formula were proved
in~\cite{CK2} for cluster categories associated with
any acyclic quiver, and later in~\cite{Pcc} for Hom-finite
2-Calabi--Yau triangulated categories. A generalisation to the
Hom-infinite case can be found in~\cite{PlamondonHomInfinite},
and a version for quantum cluster algebras in~\cite{Qin}.
The first generalization of the second multiplication formula,
by A. Hubery (see~\cite{HubCluster}), was based on the existence of Hall polynomials
which he proved in the affine case~\cite{HubHallPoly},
generalizing Ringel's result~\cite{RingelHallPoly}
for Dynkin quivers. Staying close to this point of view,
J. Xiao and F. Xu proved in~\cite{XXGreen} a projective version of
Green's formula~\cite{RingelGreen} and applied it to generalize the multiplication formula
for acyclic cluster algebras. Another proof of this formula
was found by F. Xu in~\cite{Xu}, who used the 2-Calabi--Yau property instead
of Green's formula.
Our aim in this paper is to generalize the second multiplication formula
to more general 2-Calabi--Yau categories for the
cluster character associated with an arbitrary
cluster tilting object. This in particular applies to
the generalized cluster categories introduced by C.~Amiot~\cite{Acqw}
and to stable categories of modules over a preprojective
algebra.

Assume that the triangulated category $\cat$
is the stable category
of a Hom-finite Frobenius category $\ec$. Then C. Fu
and B. Keller defined a cluster character on $\ec$,
which "lifts" the one on $\cat$. We prove that it
satisfies the same multiplication formula as
the one proved by Geiss--Leclerc--Schr\"oer
in \cite{GLSmf}.

The paper is organized as follows:
In the first section, we fix some notations and state
our main result: A multiplication formula for
the cluster character associated with any cluster tilting object.
In section~\ref{section: constr}, we recall some
definitions and prove the `constructibility
of kernels and cokernels' in modules categories.
We apply these facts to prove that:
\begin{itemize}
 \item If the triangulated category has constructible cones
(see section~\ref{ssection: mf}),
the sets under consideration in the multiplication formula, and in
its proof, are constructible.
 \item Stable categories of Hom-finite Frobenius categories
have constructible cones.
 \item Generalized cluster categories defined in~\cite{Acqw}
have constructible cones.
\end{itemize}
Thus, all of the Hom-finite
2-Calabi--Yau triangulated categories related
to cluster algebras which have been introduced so far
have constructible cones. Notably this holds for
cluster categories associated with acyclic quivers,
and for the stable categories associated with the exact
subcategories of module categories over preprojective
algebras constructed in~\cite{GLSflag} and ~\cite{BIRS}.
In section \ref{section: proof mf}, we prove the main theorem.
In the last section, we consider the setup of Hom-finite
Frobenius categories. We prove a multiplication formula
for the cluster character defined by Fu--Keller in \cite{FK}.

\tableofcontents

\section{Notations and main result}\label{section: notations}

Let $k$ be the field of complex numbers.
The only place where we will need more than
the fact that $k$ is
an algebraically closed field
is proposition~\ref{prop: Dim}
in section~\ref{ssection: def constr}.
See~\cite[section 3.3]{JoyceConstr} for an explanation,
illustrated with an example, of the fact that
the theory of constructible functions does not extend to
fields of positive characteristic.
Let $\cat$ be a Hom-finite, 2-Calabi--Yau,
Krull--Schmidt $k$-category
which admits a basic cluster tilting object $T$.
In order to prove the main theorem,
a constructibility hypothesis will be needed.
This hypothesis is precisely stated in section~\ref{ssection: H}
and it will always be explicitly stated when it is assumed.
Stable categories of Hom-finite Frobenius categories
satisfy this constructibility hypothesis,
cf. section~\ref{ssection: H stable}, so that the main theorem
applies to cluster
categories (thanks to the construction in~\cite[Theorem 2.1]{GLSterminal}),
to stable module categories over preprojective algebras...
Moreover, the main theorem applies to the generalized cluster categories
of~\cite{Acqw}, cf. section~\ref{ssection: H cqw}.

We let $B$ denote the endomorphism algebra
of $T$ in $\cat$, and we let $F = \cat(T,?)$ denote
the covariant functor from $\cat$ to $\modb$
co-represented by $T$.
We denote the image in $\qb(x_1,\ldots,x_n)$ of an object $M$ in $\cat$ under
the cluster character associated with $T$ (see~\cite{Pcc}) by $X^T_M$.
Before recalling the formula for $X^T_M$,
we need to introduce some notation.
Let $Q_T$ be the Gabriel quiver of $B$, and denote
by $1,\ldots,n$ its vertices. For each vertex $i$,
denote by $S_i$ (resp. $P_i$) the corresponding simple
(resp. projective) module.
For any two finite-dimensional $B$-modules $L$ and $N$,
define
\begin{eqnarray*}
\langle L,N \rangle_{\;\,} & = & \dim \homb(L,N) - \dim \ext^1_B(L,N),\\
\langle L,N \rangle_a & = & \langle L,N \rangle - \langle N,L \rangle.
\end{eqnarray*}
As shown in~\cite[Section 3]{Pcc}, the form $\langle \;,\; \rangle_a$
descends to the Grothendieck group $\ko$
(that is, it only depends on the dimension vectors of $L$ and $N$).
Note that this would not be true for the form $\langle \;,\; \rangle$
in general, since $B$ is quite often of infinite global dimension
(see~\cite{KR1}). For a $B$-module $L$, the projective variety
$\gre L$ is the Grassmannian of submodules
of $L$ with dimension vector $e\in\ko$.
For any object $X\in\cat$, there are triangles
$$
T_\beta \fl T_\al \fl X \fl \susp T_\beta
\; \text{ and } \;
\susp T_\delta \fl X \fl \susp^2 T_\gamma \fl \susp^2 T_\delta,
$$
with $T_\al,\ldots,T_\delta$ in add$\,T$ (see~\cite[Proposition 2.1]{KR1}),
which are triangulated analogues of projective presentations and injective
copresentations respectively.
The index and coindex of $X$ (with respect to $T$) are the following classes
in $\kop$:
$$
\ind X = \ind_T X = [FT_\al] - [FT_\beta]
\; \text{ and } \;
\coi X = \coi_T X = [FT_\gamma] - [FT_\delta].
$$
For some properties of the index, see~\cite{Pcc}, 
and for a more thorough study, see~\cite{DK}, \cite{FK}
and~\cite[Section 4.2]{PlamondonClusterAlgebras}.
Then we have:
$$
X^T_M = \xs^{-\coi M} \sum_e \chi(\gre FM) \prod_{i=1}^n x_i^{\langle S_i,e\rangle_a},
$$
where the sum runs over all classes $e\in\ko$.
For any two objects $L$ and $M$ in $\cat$, and any morphism $\eps$
in $\clm$, we denote the isomorphism class of
objects $Y$
appearing in a triangle of the form
$$M\gfl Y\gfl L\stackrel{\eps}{\gfl} \susp M$$
by $\mte$ (the middle term of $\eps$).
Note that any two such objects $Y$ are isomorphic.

\subsection{$X^T$-stratification}

Let $L$ and $M$ be objects in $\cat$.
If an object $Y$ of $\cat$ belongs to the class $\mte$
for some morphism $\eps$ in $\clm$,
we let $\class$ denote the set of all
isomorphism classes of objects $Y'\in\cat$
such that:
\begin{itemize}
 \item $Y'$ is the middle term of some morphism in $\clm$,
 \item $\coi Y' = \coi Y$ and
 \item for all $e$ in $\ko$, we have $\chi\big{(}\gre(FY')\big{)}
= \chi\big{(}\gre(FY)\big{)}$.
\end{itemize}
The equality of classes $\class=\langle Y'\rangle$
yields an equivalence relation on the `set'
of middle terms of morphisms in $\clm$.
Fix a set $\yc$ of representatives for this relation.
Further, we denote the set of all $\eps$
with $\mte\in\class$ by $\clmy$, and
the set of $\eps'\in\clm$
such that $X^T_{\operatorname{mt}(\eps')}=X^T_{\mte}$
by $\langle\eps\rangle$.
It will be proven in section~\ref{ssection: clmy}
that if the cylinders of the morphisms $L\fl\susp M$
are constructible with respect to $T$
in the sense of section~\ref{ssection: H} below,
then the sets $\clmy$ are constructible, and
the set $\yc$ is finite.

Remark that if $Y'$ belongs to $\class$, then $X^T_{Y'} = X^T_{Y}$.
Hence the fibers of the map sending $\eps$ to $X^T_{\mte}$
are finite unions of sets $\clmy$.
Therefore, the sets $\langle\eps\rangle$ are constructible;
we have
$$
\clm = \coprod_{\eps\in\rc}\langle\eps\rangle
$$
for some finite set $\rc\subset\clm$, and
$$
\clm = \coprod_{Y\in\yc}\clmy
$$
is a refinement of the previous decomposition.

\subsection{The variety $\rep_dBQ$}

Let $V$ be a finite dimensional
$k$-vector space. We denote by $\rep_B'(V)$
the set of morphisms of $k$-algebras from $B^\text{op}$
to $\End_k(V)$.
Since $B$ is finitely generated,
the set $\rep_B'(V)$ is a closed subvariety
of some finite product of copies of $\End_k(V)$.

Let $Q$ be a finite quiver, and let $d = (d_i)_{i\in Q_0}$
be a tuple of non-negative integers. A $d$-\emph{dimensional
matrix representation} of $Q$ in $\modb$ is given by
\begin{itemize}
 \item a right $B$-module structure on $k^{d_i}$ for
each vertex $i$ of $Q$ and
 \item a $B$-linear map $k^{d_i}\fl k^{d_j}$ for
each arrow $\al : i\fl j$ of $Q$.
\end{itemize}
Clearly, for fixed $d$, the $d$-dimensional matrix representations
of $Q$ in $\modb$ form an affine variety $\rep_d BQ$ on which
the group $GL(d)=\prod_{i\in Q_0}GL_{d_i}(k)$ acts by changing
the bases in the spaces $k^{d_i}$. We write $\rep_d BQ / GL(d)$
for the set of orbits.

\subsection{Constructible cones}\label{ssection: H}

Let $\overrightarrow{A_4}$ be the quiver:
$1\fl 2\fl 3\fl 4$. Let $T$, $L$ and $M$ be objects of
$\cat$. Let $d_\text{max}$ be the $4$-tuple of integers
$$(\dim FM, \dim FM+\dim FL, \dim FL, \dim F\susp M).$$
Let $\Phi_{L,M}$ be the map from $\clm$ to
$$
\coprod_{d\leq d_\text{max}}\rep_d (B\overrightarrow{A_4})/GL(d)
$$
sending a morphism $\eps$ to the orbit of the exact sequence
of $B$-modules
$$\xymatrix@-.1pc{\cat(T,M) \dr^{Fi} & \cat(T,Y) \dr^{Fp} &
\cat(T,L) \dr^{F\eps\;\;\;} & \cat(T,\susp M),}$$
where
$M\stackrel{i}{\gfl} Y \stackrel{p}{\gfl} L \stackrel{\eps}{\gfl} \susp M$
is a triangle in $\cat$.
The cylinders over the morphisms $L\fl\susp M$ are
\emph{constructible with respect to $T$} if
the map $\Phi_{L,M}$ lifts to a constructible
map
$$
\vph_{L,M} : \clm \gfl \coprod_{d\leq d_\text{max}} \rep_d(B\overrightarrow{A_4})
$$
(see section~\ref{ssection: def constr}).
The category $\cat$ is said to have constructible cones
if this holds
for arbitrary objects $L$,$M$ and $T$.

\subsection{Main result}\label{ssection: mf}

Let $f$ be a constructible function from an algebraic variety over $k$
to any abelian group,
and let $C$ be a constructible subset of this variety. Then one defines
``the integral of $f$ on $C$ with respect to the Euler characteristic'' to be
$$
\int_C f = \sum_{x\in f(C)}\chi\big(C\cap f^{-1}(x) \big)x,
$$
cf. for example the introduction of~\cite{Lconstr}.
Our aim in this paper is to prove the following:
\begin{theo}\label{theo: mf} Let $T$ be any
cluster tilting object in $\cat$. Let  $L$ and $M$ be
two objects such that the cylinders over the morphisms
$L\fl\susp M$ and $M\fl\susp L$ are constructible
with respect to $T$. Then we have:
$$
\chi(\pclm)X^T_L X^T_M = \int_{[\eps] \in \pclm} X^T_{\mte} 
+ \int_{[\eps] \in \pcml} X^T_{\mte},
$$
where $[\eps]$ denotes the class in $\pclm$
of a non zero morphism $\eps$ in $\clm$.
\end{theo}
The statement of the theorem is inspired from~\cite{GLSmf},
cf. also~\cite{XXGreen}. We will prove it in section~\ref{section: proof mf}.
Our proof is inspired from that of P. Caldero and B. Keller
in~\cite{CK1}. Note that in contrast with
the situation considered there, in the above formula,
an infinite number of isomorphism classes
$\mte$ may appear.

\section{Constructibility}\label{section: constr}

\subsection{Definitions}\label{ssection: def constr}

Let $X$ be a topological space.
A \emph{locally closed subset} of $X$ is the intersection of
a closed subset with an open one.
A \emph{constructible subset} is a finite (disjoint) union of
locally closed subsets. The family of constructible
subsets is the smallest one containing all open
(equivalently: closed) subsets
of $X$ and stable under taking finite intersections
and complements.
A function $f$ from $X$ to 
an abelian group is \emph{constructible} if it is
a finite $\zb$-linear combination of characteristic
functions of constructible subsets of $X$.
Equivalently, $f$ is constructible if it takes a
finite number of values and if its fibers are
constructible subsets of $X$.

For an algebraic variety $X$,
the ring of constructible functions from $X$ to $\zb$
is denoted by $CF(X)$.
The following proposition will be used,
as in~\cite{XXGreen}, in order to prove
lemma~\ref{lem: mue} of section~\ref{ssection: clmy}.

\begin{prop}\label{prop: Dim}
\cite[Proposition 4.1.31]{Dim}
Associated with any morphism of complex
algebraic varieties
$f:X\gfl Y$, there is a well-defined
push-forward homomorphism
$CF(f):CF(X)\gfl CF(Y)$. It is determined
by the property
$$
CF(f)(1_Z)(y) = \chi(f^{-1}(y)\cap Z)
$$
for any closed subvariety $Z$ in $X$
and any point $y\in Y$.
\end{prop}

Let $X$ and $Y$ be algebraic varieties.
A map $f:X\gfl Y$ is
said to be \emph{constructible} if
there exists a decomposition of $X$ into a finite disjoint union
of locally closed subsets $X_i, i\in I$, such that
the restriction of $f$ to each $X_i$ is a morphism of algebraic varieties.
Note that the composition of two constructible maps is constructible, and that
the composition of a constructible function with a constructible map
is again a constructible function.

\subsection{Kernels and cokernels are constructible}\label{ssection: coker}

In section 2.1 of~\cite{Xu}, it is shown that
the kernel and cokernel of a morphism of modules over a path algebra
$\cb Q$ are constructible. In this section,  we give direct
proofs in the more general case where $\cb Q$ is replaced by a
finite dimensional algebra $B$.

Let $L$ and $M$ be two finite dimensional vector spaces
over the field $k$, of respective dimensions $n$ and $m$.
Let $N$ be a linear subspace of $M$.
Define $E_N$ to be the set of all morphisms $f\in\homph_k(L,M)$
such that $\im f \oplus N = M$.

\begin{lem}\label{lem: locally closed}
The set $E_N$
is a locally closed subset of $\homph_k(L,M)$.
\end{lem}

\begin{proof}
Let $(u_1,\ldots,u_n)$ be a basis of $L$, and
let $(v_1,\ldots,v_m)$ be a basis of $M$ whose
$p$ first vectors form a basis of $N$.
Let $r$ be such that $r+p=m$.
Let $f:L\gfl M$ be a $k$-linear map, and denote by
$A=(a_{ij})$ its matrix in the bases
$(u_1,\ldots,u_n)$ and $(v_1,\ldots,v_m)$.
Denote by $A_1$ the submatrix of $A$ formed by
its first $p$ rows and by $A_2$ the one formed
by its last $r$ rows. For $t\leq n$, let
$P(t,n)$ be the set of all subsets of $\ens{1,\ldots,n}$
of cardinality $t$.

The map $f$ belongs to $E_N$ if and only if:
\begin{itemize}
 \item[a)] There exists $\js$ in $P(r,n)$
such that the submatrix $(a_{ij})_{i>p,j\in\js}$
has a non-zero determinant and
 \item[b)] if the last $r$ entries of a linear combination
of columns of $A$ vanish, then the combination itself vanishes.
\end{itemize}
Condition b) is equivalent to the inclusion
$\ke A_2 \subseteq \ke A_1$ and so to
the inclusion $\im (A_1^\text{t})\subseteq\im (A_2^\text{t})$.
Therefore, condition b) can be restated as condition b'):
\begin{itemize}
 \item[b')] For all $i_0\leq p$, and all $\ls\in P(r+1,n)$,
the determinant of the submatrix of $A$ obtained by taking
lines in $\ens{i_0,p+1,\ldots,m}$ and columns in $\ls$
vanishes.
\end{itemize}
Let $\Omega_{\js}$ be the set of all maps
that satisfy condition a) with respect to the index
set $\js$, and let $F$ be the set of
all maps that satisfy condition b').
For all $\js\in P(r,n)$, the set $\Omega_{\js}$ is
an open subset of $\homph_k(L,M)$ and the set
$F$ is a closed subset of $\homph_k(L,M)$.
Since we have the equality:
$$
E_N = \big(\bigcup_{\js\in P(r,n)}\Omega_{\js}\big)\cap F,
$$
the set $E_N$ is locally closed in $\homph_k(L,M)$.
\end{proof}

Let $\overrightarrow{A_2}$ be the quiver:
$1 \fl 2$.
\begin{lem}\label{lem: coker}
Let $B$ be a finite dimensional algebra, and
let $L$ and $M$ be finitely generated $B$-modules
of dimensions $n$ and $m$ respectively. The
map $c$ from $\homb(L,M)$ to
$\coprod_{d\leq m}\rep_{(m,d)} (B\overrightarrow{A_2})/GL(m,d)$
which sends a morphism $l$ to the orbit of the representation
$\xymatrix@-.4pc{M\dre & \coke l}$ lifts to a constructible
map from $\homb(L,M)$ to $\coprod_{d\leq m}\rep_{(m,d)} (B\overrightarrow{A_2})$.

Dually, the map from
$\homb(L,M)$ to $\coprod_{d\leq n}\rep_{(d,n)} (B\overrightarrow{A_2})/GL(d,n)$
which sends a morphism $l$ to the orbit of the representation
$\xymatrix@-.7pc{\ke l\;\; \dri & L}$ lifts to a constructible
map from $\homb(L,M)$ to $\coprod_{d\leq n}\rep_{(d,n)} (B\overrightarrow{A_2})$.
\end{lem}

\begin{proof}
Let us prove the first assertion.
We keep the notations of the proof of
lemma~\ref{lem: locally closed}.
For a subset $\is$ of $\ens{1,\ldots,m}$,
let $N_{\is}$ be the linear subspace of $M$
generated by $(v_i)_{i\in\is}$.
Then $\homb(L,M)$ is the union
of its intersections with each
$E_{N_{\is}}$, for~${\is\subseteq\ens{1,\ldots,m}}$. It is thus enough to consider
the restriction of the map $c$ to $E_N$,
where $\xymatrix@-.6pc{N\;\dri^{i_N} & M}$
is a given linear subspace of $M$.
Since the set $E_N$ is the union of
the locally closed subsets $\Omega_{\js}\cap F$,
for $\js\in P(r,n)$, we can fix such a $\js$ and
only consider the restriction of
$c$ to $\Omega_{\js}\cap F$.
Let $f$ be a morphism in $\homb(L,M)$ and
assume that $f$ is in $\Omega_{\js}\cap F$.
Then the cokernel of the $k$-linear map $f$ is $N$
and the projection $p_f$ of $M$ onto $N$ along $\im f$
is given by the $n\times p$ matrix
$\left( 1 \;\; -CD^{-1}\right)$, where
$C$ is the submatrix $(a_{ij})_{i\leq p, j\in\js}$ and
$D$ is the submatrix $(a_{ij})_{i > p, j\in\js}$.
Moreover, if we denote by $\rho^M\in\rep_B'(M)$
the structure of $B$-module of $M$, then the structure
of $B$-module $\rho$ of $N$ induced by $f$
is given by
$\rho(b) = p_f\circ\rho^M(b)\circ i_N$, for all $b\in B$.
\end{proof}

\subsection{Constructibility of $\clmy$}\label{ssection: clmy}

Let $k$, $\cat$ and $T$ be as in section~\ref{section: notations}.
Recall that $B$ denotes the endomorphism algebra
$\End_\cat(T)$. This algebra is the path algebra
of a quiver $Q_T$ with ideal of relations $I$.
Recall that we denote by $1,\ldots,n$ the
vertices of $Q_T$.

The following lemma is a particular case of
\cite[Proposition 4.1.31]{Dim}, and was already
stated in~\cite{XXGreen} for hereditary algebras.

\begin{lem}\label{lem: mue}
For any two dimension vectors $e$ and $d$ with
$e\leq d$, the function
\begin{eqnarray*}
\mu_e : \rep_d(Q_T,I) & \gfl & \zb \\
M & \longmapsto & \chi(\gre M)
\end{eqnarray*}
is constructible.
\end{lem}

\begin{proof}
Let $\gre(d)$ be the closed subset of
$$\rep_d(Q_T,I)\times\prod_{i\in Q_0}\gr_{e_i}(k^{d_i})$$
formed by those pairs $(\rho,W)$ for which
the subspaces $W_i\subseteq k^{d_i}$, $i\in Q_0$,
form a subrepresentation.
Apply proposition~\ref{prop: Dim} to
the first projection
${f:\gre(d)\fl\rep_d(Q_T,I)}$
and remark that
$\mu_e = CF(f)(1_{\gre(d)})$.

\end{proof}

\begin{cor}\label{cor: chi constr}
Let $L$ and $M$ be objects in $\cat$,
and let $e\leq \vdim FL + \vdim FM$ be in $\ko$.
Assume that the cylinders over the morphisms
$L\fl \susp M$ are constructible.
Then the function
\begin{eqnarray*}
\ld_e : \; \clm & \gfl &  \zb \\
\eps & \longmapsto & \chi(\gre F\mte)
\end{eqnarray*}
is constructible.
\end{cor}

\begin{proof}
By our hypothesis, the map sending $\eps\in\clm$
to the image of its middle term in
$\coprod\rep_d(Q_T,I)/GL(d)$, where
the union is over the dimension vectors $d$
not greater than $\vdim FL + \vdim FM$,
lifts to a constructible map
from $\clm$ to $\coprod\rep_d(Q_T,I)$. The claim therefore
follows from lemma~\ref{lem: mue}.
\end{proof}

Let
$M\stackrel{i}{\gfl} Y \stackrel{p}{\gfl} L \stackrel{\eps}{\gfl} \susp M$
be a triangle in $\cat$,
and denote
by $g$ the class of $\ke Fi$ in the Grothendieck group $\ko$.

\begin{lem}\label{lem: coind mt}
We have:
$$
\coi Y = \coi (L\oplus M) - \sum_{i= 1}^n \langle S_i, g\rangle_a [P_i].
$$
\end{lem}

\begin{proof}
Let $K\in\cat$ lift $\ke Fi$.
Using respectively proposition~2.2, lemma~2.1.(2), lemma~7 and section~3
of \cite{Pcc}, we have the following equalities:
\begin{eqnarray*}
 \coi Y & = & \coi L + \coi M - \coi K - \coi \susp K \\
 & = & \coi (L\oplus M) + \ind K - \coi K \\
 & = & \coi (L\oplus M) - \sum_{i= 1}^n \langle S_i, FK\rangle_a [P_i] \\
 & = & \coi (L\oplus M) - \sum_{i= 1}^n \langle S_i, g\rangle_a [P_i].
\end{eqnarray*}
\end{proof}

\begin{cor}\label{cor: coind constr}
Let $L$ and $M$ be two objects such that
the cylinders over the morphisms $L\fl\susp M$
are constructible.
The map $\ld : \clm \gfl \kop$
which sends $\eps$ to the coindex
(or to the index) of
its middle term $Y$ is constructible.
\end{cor}

\begin{proof}
Note that $g$ is at most the sum of the
dimension vectors of $FL$ and $FM$, so that
by lemma~\ref{lem: coind mt} the map $\ld$
takes a finite number of values.
By our hypothesis and
lemma~\ref{lem: coker}, there exists a constructible map:
$$
\clm \gfl \coprod_{d\leq \dim FM} \rep_B'(k^d)
$$
which lifts the map sending $\eps$ to the isomorphism class
of the structure of $B$-module
on $\ke Fi$. Moreover, the map sending
a module $\rho$ in $\bigcup_{d\leq \dim FM} \rep_B'(k^d)$
to $\sum_{i= 1}^n \langle S_i, \rho\rangle_a [P_i]$ in $\kop$
only depends on the dimension vector of $\rho$ and
thus is constructible. Therefore, the map
$\ld$ is constructible.
\end{proof}

\begin{prop}\label{prop: clmy constr}
Let $L,M\in\cat$ be such that
the cylinders over the morphisms $L\fl\susp M$
are constructible.
Then the sets $\clmy$ are constructible subsets of $\clm$.
Moreover, the set $\clm$ is a finite disjoint
union of such constructible subsets.
\end{prop}

\begin{proof}
Fix a triangle
$M\stackrel{i}{\gfl} Y \stackrel{p}{\gfl} L \stackrel{\eps}{\gfl} \susp M$
in $\cat$. Then $\eps'\in\clm$ is in $\clmy$ if and only if
\begin{itemize}
 \item $\ld(\eps')=\ld(\eps)$ and
 \item For all $e\leq \vdim FY$, $\ld_e(\eps')=\ld_e(\eps)$.
\end{itemize}
Therefore, the claim follows from corollary~\ref{cor: chi constr}
and corollary~\ref{cor: coind constr}.

\end{proof}

\subsection{Stable categories have constructible cones}\label{ssection: H stable}

In this section, we assume moreover that $\cat$ is
the stable category of a
Hom-finite, Frobenius, Krull--Schmidt category $\ec$,
which is
linear over the algebraically closed field $k$.
Our aim is to prove that such a category has constructible cones.

Let $\pc$ denote the ideal in $\ec$ of morphisms
factoring through a projective-injective object.
Let $T$, $L$ and $M$ be objects of the
category $\cat$.
Fix a $k$-linear section $s$ of the projection
$\xymatrix@-.4pc{\elm \dre & \clm}$ induced
by the canonical functor $\ec \stackrel{\Pi}{\gfl} \cat$.
Fix a conflation
$\xymatrix@-.4pc{M\;\, \dri & IM\dre &\susp M}$ in $\ec$,
with $IM$ being projective-injective in $\ec$,
and, for any $\eps$ in $\clm$,
consider its pull-back via $s\eps$:
$$
\xymatrix{
M\; \baseg \dri^\iota & Y \dre^\pi \bas & L \bas^{s\eps} \\
M\; \dri & IM\dre & \susp M.
}
$$
Via $\Pi$, this diagram induces a triangle
$
M\stackrel{i}{\gfl} Y \stackrel{p}{\gfl} L \stackrel{\eps}{\gfl} \susp M
$
in $\cat$.

For any $X\in\ec$, we have a
commutative diagram with exact rows:
$$
\xymatrix{
0 \dr & \ec(X,M)\dr^{\ec(X,\iota)}\baseg & \ec(X,Y)\dr^{\ec(X,\pi)} \bas &
\ec(X,L)\bas^{\ec(X,s\eps)}  \\
0\dr & \ec(X,M)\dr & \ec(X,IM)\dr & \ec(X,\susp M) .
}
$$

Fix $X'\in\ec$ and a morphism $X'\fl X$.
Denote by $C$ the endomorphism algebra
of $X'\fl X$ in the category of morphisms of $\ec$, and by
$\dc'$ the set of dimension vectors
$d=(d_1,d_2,d_3,d_4)$ such that
$d_1=\dim \ec(X,M)$, $d_3=\dim \ec(X,L)$,
$d_2\leq d_1 + d_3$ and $d_4 = \dim\ec(X,\susp M)$.

\begin{lem}\label{lem: mu constr}
There exists a constructible map
$$\mu : \clm \gfl \coprod_{d\in\dc'}\rep_d C\overrightarrow{A_4}$$
which lifts the map sending $\eps$ to the orbit of the
matrix representation of $\overrightarrow{A_4}$ in $\operatorname{mod}C$
given by
$\xymatrix{\ec(X,M) \dr^{\ec(X,\iota)} & \ec(X,Y) \dr^{\ec(X,\pi)} &
\ec(X,L) \dr^{\ec(X,s\eps)\;\;} & \ec(X,\susp M)}$.
\end{lem}

\begin{proof}
By definition of a pull-back,
the map $\ec(X,Y) \gfl \ec(X,IM)\oplus\ec(X,L)$
is a kernel for the map $\ec(X,IM)\oplus\ec(X,L) \gfl \ec(X,\susp M)$,
with appropriate signs.
Moreover, the morphism $\xymatrix{\ec(X,M)\dr^{\ec(X,\iota)}&\ec(X,Y)}$
is a kernel for $\ec(X,\pi)$.
Therefore, lemma~\ref{lem: coker} in section~\ref{ssection: coker}
applies and such a constructible map $\mu$ exists.
\end{proof}

Denote by $\dc$ the set of dimension vectors
$d=(d_1,d_2,d_3,d_4)$ such that:

\noindent $d_1 = \dim\cat(T,M)$, $d_3 = \dim\cat(T,L)$,
$d_2\leq d_1 + d_3$ and $d_4 = \dim\cat(T,\susp M)$.

\begin{prop}\label{prop: phi constr}
There exists a constructible map
$$
\vph : \clm \gfl \coprod_{d\in\dc}\rep_d B\overrightarrow{A_4}
$$
which lifts the map sending $\eps$ to the orbit
of the representation
$$\xymatrix@-.1pc{\cat(T,M) \dr^{Fi} & \cat(T,Y) \dr^{Fp} &
\cat(T,L) \dr^{F\eps} & \cat(T,\susp M)}.$$
\end{prop}

\begin{proof}
Let $T\infl IT$ be an inflation from $T$ to a projective-injective
object in $\ec$.
This inflation induces a commutative diagram $(\ast)$
of modules over the endomorphism algebra $\widetilde{B}$ of
$T\infl IT$ in the Frobenius category of inflations of $\ec$:
$$
\xymatrix{
\ec(IT,M) \dr\bas_{(\ast)\;\;\;\;\;\;\;\;\;\;\;\;\;\;\;\;}
& \ec(IT,Y)\dr\bas & \ec(IT,L)\dr\bas & \ec(IT,\susp M) \bas \\
\ec(T,M) \dr & \ec(T,Y)\dr & \ec(T,L)\dr & \ec(T,\susp M).
}
$$
The map which sends
$\eps$ to the orbit of the diagram $(\ast)$ lifts to a constructible one.
This is proved by repeating the proof of lemma~\ref{lem: mu constr}
for the functor
$$
\ec\gfl \operatorname{mod}\widetilde{B},\;
U\longmapsto \big(\ec(IT,U)\fl\ec(T,U)\big)
$$
instead of $U\mapsto\ec(X,U)$
and using lemma~\ref{lem: coker} for $\widetilde{B}$.

By applying lemma~\ref{lem: coker} to
$\widetilde{B}\otimes kA_4$, we see that the vertical
cokernel of diagram $(\ast)$ is constructible as a
$\widetilde{B}\otimes kA_4$-module.
Now the claim follows because the terms of the
cokernel are $B$-modules and $B$ is also
the stable endomorphism algebra of
$T\infl IT$ in the Frobenius category of inflations of $\ec$.
\end{proof}

\subsection{Generalized cluster categories have constructible cones}\label{ssection: H cqw}

Let $(Q,W)$ be a Jacobi-finite quiver with potential $W$ in $kQ$
(cf. section~3.3 of~\cite{Acqw}),
and let $\Gamma$ be the Ginzburg dg algebra associated with $(Q,W)$
(cf. section~4.2 of~\cite{Ginzburg}).
The perfect derived category $\perg$ is the thick subcategory
of the derived category $\dc\Gamma$ generated by $\Gamma$.
The finite dimensional derived category
$\dc_\text{fd}\Gamma$
is the full subcategory of $\dc\Gamma$ whose objects are
the dg modules whose homology is of finite total dimension.
An object $M$ belongs to $\dc_\text{fd}\Gamma$
if and only if
$\homph_{\dc\Gamma}(P,M)$ is finite dimensional for each
object $P$ of $\perg$.

\vspace{3mm}
\noindent {\bf Lemma} [Keller--Yang]

\begin{itemize}
 \item[a)] The category $\dc_\text{fd}\Gamma$ is contained in $\perg$.
 \item[b)] An object of $\dc\Gamma$ belongs to $\dc_\text{fd}\Gamma$
if and only if it is quasi-isomorphic to a dg $\Gamma$-module
of finite total dimension.
 \item[c)] The category $\dc_\text{fd}\Gamma$ is equivalent to
the localization of the homotopy category $\hc_\text{fd}\Gamma$
of right dg $\Gamma$-modules of finite total dimension
with respect to its subcategory of acyclic dg modules.
\end{itemize}
\vspace{3mm}

Note that the previous lemma,
taken from the appendix of~\cite{KY}, is stated above under
some restrictions which do not appear there.

Recall that the generalized cluster category associated with $(Q,W)$,
defined in~\cite{Acqw},
is the localization of the category $\operatorname{per}\Gamma$ by
the full subcategory $\dc_\text{fd}\Gamma$.
It is proved in~\cite{Acqw} that the canonical
t-structure on $\dc\Gamma$ restricts to a t-structure
on $\perg$. We will denote this t-structure
by $(\per^{\leq 0},\per^{\geq 0})$.
Denote by $\fc$ the full
subcategory of $\perg$ defined by:
$$
\fc = \per^{\leq 0}\cap\,^\perp(\per^{\leq -2}).
$$
Recall from \cite{Acqw} that the canonical functor
from $\perg$ to $\cat_\Gamma$ induces a $k$-linear equivalence
from $\fc$ to $\cat_\Gamma$ and that the functor
$\tau_{\leq -1}$ induces an equivalence from
$\fc$ to $\susp\fc$.

Fix an object $T$ in $\cat_\Gamma$. Without loss
of generality, assume that $T$ belongs to $\fc$.
Note that the canonical cluster tilting object
$\Gamma\in\cat_\Gamma$ does belong to $\fc$.

\begin{lem}\label{lem: perg}
Let $X$ be an object of $\perg$. If $X$
is left orthogonal to $\per^{\leq -3}$,
which happens for instance when $X$ is in $\fc$ or in $\susp\fc$,
then there is a functorial isomorphism
$$
\homph_{\perg}(\tau_{\leq -1}T,X)\stackrel{\simeq}{\gfl}
\cat_\Gamma(T,X).
$$
\end{lem}

\begin{proof}
Let $X\in\perg$ be left orthogonal to $\per^{\leq -3}$.
By~\cite[Proposition 2.8]{Acqw}, we have
$\cat_\Gamma(T,X) = \underrightarrow{\lim}
\homph_{\perg}(\tau_{\leq n}T,\tau_{\leq n}X).$
Moreover, for any $n$, we have
$$
\homph_{\perg}(\tau_{\leq n}T,\tau_{\leq n}X) = \homph_{\perg}(\tau_{\leq n}T,X).
$$

Let $n<-1$.
The object $\tau_{[n+1,-1]}T$ belongs to $\dc_\text{fd}(\Gamma)$
and $X$ belongs to $\perg$,
so that the $3$-Calabi--Yau property (see~\cite{KCY}) implies
that the morphism space $\homph_{\perg}(\susm\tau_{[n+1,-1]}T,X)$ is isomorphic
to the dual of $\homph_{\perg}(X,\susp^2\tau_{[n+1,-1]}T)$.
This latter vanishes since $X$ belongs to $^\perp(\per^{\leq -3})$.
The same argument shows that the space
$\homph_{\perg}(\tau_{[n+1,-1]}T,X)$ also vanishes.
Therefore applying the functor $\homph_{\perg}(?,X)$ to
the triangle
$$
\susm \tau_{[n+1,-1]}T \gfl \tau_{\leq n}T \gfl
\tau_{\leq -1}T \gfl \tau_{[n+1,-1]}T,
$$
yields an isomorphism
$\homph_{\perg}(\tau_{\leq n}T,X)\stackrel{\simeq}{\gfl}
\homph_{\perg}(\tau_{\leq -1}T,X)$.
%
\end{proof}

\begin{lem}\label{lem: fdg}
Let $X,Y\in\perg$ and assume that $X$ belongs to $^\perp(\per^{\leq-3})$.
Then the functor $\tau_{\geq-2}$ induces a bijection
$\homph_{\perg}(X,Y) \simeq \homph_{\dc_\text{\emph{fd}}(\Gamma)}
(\tau_{\geq-2}X,\tau_{\geq-2}Y).$
\end{lem}

\begin{proof}
By assumption, $X$ is left orthogonal to the subcategory
$\per^{\leq-3}$. Therefore,
the space $\homph_{\perg}(X,Y)$ is isomorphic to
$\homph_{\perg}(X,\tau_{\geq-2}Y)$, and thus to
$\homph_{\perg}(\tau_{\geq-2}X,\tau_{\geq-2}Y)$.
Since $X$ and $Y$ are perfect over $\Gamma$, their images
under $\tau_{\geq-2}$ are quasi-isomorphic
to dg modules of finite total dimension.
\end{proof}

\begin{prop}\label{prop: dfd H}
Let $\Gamma$ be the Ginzburg dg algebra associated with
a Jacobi-finite quiver. Then the category
$\dc_\text{fd}(\Gamma)$ has constructible cones.
\end{prop}

\begin{proof}
We
write $\mathfrak{n}$ for the ideal of $\Gamma$
generated by the arrows of the Ginzburg quiver,
and {\bfseries p} for the left adjoint to the
canonical functor $\hc(\Gamma)\fl\dc(\Gamma)$.
Let $L$, $M$ and $T$ be dg modules
of finite total dimension. Since
$\homph_{\dc_\text{fd}(\Gamma)}(L,\susp M)$ is finite
dimensional, there exists a quasi-isomorphism
$M\stackrel{ w}{\gfl} M'$, where $M'$
is of finite total dimension and such that
any morphism $L\fl\susp M$ may be represented
by a fraction:
$$\xymatrix@-1pc{L\bdr & & \susp M. \bg^{\susp w} \\ & \susp M' &}$$
We thus obtain a surjection
$\xymatrix{\ext^1_{\hc_\text{fd}(\Gamma)}(L,M') \dre &
\ext^1_{\dc_\text{fd}(\Gamma)}(L,M).}$ Fix a $k$-linear
section $s$ of this surjection.
Choose $m$ such that $M'\mathfrak{n}^m$ and
$L\mathfrak{n}^m$ vanish. Then for the cone
$Y$ of any morphism from $\susm M'$ to $L$,
we have $Y\mathfrak{n}^m = 0$. For $X$ being any one
of $L$, $M'$, $Y$ we thus have isomorphisms
$$\homph_{\dc_\text{fd}(\Gamma)}(T,X) \simeq \homph_{\hc(\Gamma)}(\text{\bfseries p}T,X)
\simeq\homph_{\hc_\text{fd}(\Gamma)}(T',X)$$
where $T'$ denotes the finite dimensional quotient of
{\bfseries p}$T$ by $(\text{\bfseries p}T)\mathfrak{n}^m$.
The category $\hc_\text{fd}(\Gamma)$ is the stable category of
a Hom-finite Frobenius category. By section~\ref{ssection: H stable},
the category $\hc_\text{fd}(\Gamma)$ has constructible cones:
There exists a constructible map
$\varphi_{L,M'}$ (associated with $T'$)
as in section~\ref{ssection: H}. By composing
this map with the section $s$, we obtain a map
$\varphi_{L,M}$ as required.
\end{proof}

\begin{prop}
Let $\Gamma$ be the Ginzburg dg algebra associated with
a Jacobi-finite quiver. Then the generalized cluster category
$\cat_\Gamma$ has constructible cones.
\end{prop}

\begin{proof}
Let $L$ and $M$ be in $\cat_\Gamma$. Up to replacing them
by isomorphic objects in $\cat_\Gamma$, we may assume that
$L$ belongs to $\susp\fc$ and $M$ to $\fc$.
The projection then induces an isomorphism
$\homph_{\perg}(L,\susp M) \stackrel{\simeq}{\gfl}
\cat_\Gamma(L,\susp M)$. Let $\eps$ be in
$\homph_{\perg}(L,\susp M)$, and let
$M\fl Y\fl L\stackrel{\eps}{\fl}\susp M$
be a triangle in $\perg$.
Let us denote the sets of morphisms
$\homph_{\perg}(\;,\;)$ by $(\;,\;)$.
There is a commutative diagram
$$
\xymatrix@-.75pc{
(\tau_{\leq-1} T,\susm L) \bas\dr &
(\tau_{\leq-1} T,M) \bas\dr &
(\tau_{\leq-1} T,Y) \bas\dr &
(\tau_{\leq-1} T,L) \bas\dr &
(\tau_{\leq-1} T,\susp M) \bas \\
\cat_\Gamma(T,\susm L) \dr & \cat_\Gamma(T,M) \dr &
\cat_\Gamma(T,Y) \dr & \cat_\Gamma(T,L) \dr &
\cat_\Gamma(T,\susp M),
}
$$
where the morphisms in the first two and in the last two columns
are isomorphisms by lemma~\ref{lem: perg}, and so is the middle
one by the five lemma.
Note that $\tau_{\leq-1}T$ belongs to $\susp\fc$,
so that, by lemma~\ref{lem: fdg}, we have isomorphisms:
$$\homph_{\perg}(L,\susp M) \simeq \homph_{\dc_\text{fd}(\Gamma)}
(\tau_{\geq-2}L,\tau_{\geq-2}\susp M)$$ and
$$\cat_\Gamma(T,X) \simeq
\homph_{\dc_\text{fd}(\Gamma)}
(\tau_{[-2,-1]}T,\tau_{\geq-2}X)$$ for
$X\in\ens{\susm L,M,L,\susp M}$ and thus also for
$X$ being the middle term of any triangle in
$\ext^1_{\perg}(L,M)$.
Let $\eps\in\cat_\Gamma(L,\susp M)$ and let
$M\fl Y\fl L\stackrel{\eps}{\fl} \susp M$
be a triangle in $\cat_\Gamma$. Let $\overline{\eps}$
be the morphism in $\homph_{\dc_\text{fd}(\Gamma)}(\tau_{\geq -2}L,
\tau_{\geq-2}\susp M)$ corresponding to $\eps$ and let
$\tau_{\geq-2}M\fl Z\fl \tau_{\geq-2}L\stackrel{\overline{\eps}}{\fl} \tau_{\geq-2}\susp M$
be a triangle in $\dc_\text{fd}(\Gamma)$.
Then the sequence obtained from
$\susm L \fl M \fl Y\fl L \fl \susp M$
by applying the functor $\cat_\Gamma(T,?)$
is isomorphic to the one obtained from
$\susm \tau_{\geq-2}L \fl \tau_{\geq-2}M\fl Z\fl \tau_{\geq-2}L\fl \tau_{\geq-2}\susp M$
by applying the functor $\homph_{\dc_\text{fd}(\Gamma)}(\tau_{[-2,-1]}T,?)$.
By proposition~\ref{prop: dfd H},
the cylinders of the morphisms $L\fl\susp M$
are constructible with respect to $T$.
\end{proof}

\section{Proof of theorem~\ref{theo: mf}}\label{section: proof mf}

Let $T$ be a cluster tilting object of $\cat$.
Let $L$ and $M$ be two objects in $\cat$, such that
the cylinders of the morphisms $L\fl\susp M$ and $M\fl\susp L$
are constructible with respect to $T$.
Let us begin the proof with some notations
and some considerations on constructibility.
Let $\eps$ be a morphism in $\clmy$ for some $Y\in\cat$,
and let
$M\stackrel{i}{\gfl} Y' \stackrel{p}{\gfl} L \stackrel{\eps}{\gfl} \susp M$
be a triangle in $\cat$. The image of $\eps$ under
$\vph_{L,M}$ lifts the orbit of the matrix representation
of $\overrightarrow{A_4}$ in $\modb$ given by
$\xymatrix@-.1pc{FM \dr^{Fi} & FY' \dr^{Fp} &
FL \dr^{F\eps\;\;\;} & F\susp M}$ (the definition
of the constructible map $\vph_{L,M}$ is given in section~\ref{ssection: H}).
In all of this section, we will take the liberty of
denoting by $Fi$, $Fp$ and $FY'$ the image $\vph_{L,M}(\eps)$.
Denote by $\Delta$ the dimension vector $\vdim FL + \vdim FM$.
For any object $Y$ in $\cat$
and any non-negative $e$, $f$ and $g$ in $\ko$,
let $\wefg$ be the subset of
$$\pclmy\times\coprod_{d\leq \Delta}
\prod_{i=1}^n \gr_{g_i}(k^{d_i})$$
formed by the pairs $([\eps],E)$ such that
$E$ is a submodule of $FY'$ of dimension vector $g$,
$\vdim (Fp)E = e$ and $\vdim (Fi)^{-1}E = f$,
where $FY'$,$Fi$ and $Fp$ are given by $\vph_{L,M}(\eps)$.
We let
\begin{itemize}
 \item $\wg$ denote the union of all $\wefg$ with
$e\leq\vdim FL$ and $f\leq\vdim FM$ and
 \item $\wef$ denote the union
of all $\wefg$ with $g\leq \vdim FL + \vdim FM$.
\end{itemize}

\begin{lem}\label{lem: wefg constr}
The sets $\wefg$ are constructible.
\end{lem}

\begin{proof}
Denote by $\Delta$ the dimension vector
$\vdim FL + \vdim FM$, and
fix a dimension vector $g$.
Consider the map induced by $\vph_{L,M}$ which sends a pair
$(\eps,E)$ in $\clmy\times\coprod_{d\leq \Delta}\prod_{i\in Q_0}
\gr_{g_i}(k^{d_i})$ to $(Fi,Fp,FY',E)$.
By our assumption, this map (exists and) is constructible.
Therefore, the subset of
$$\clmy\times\coprod_{d\leq \Delta}\prod_{i\in Q_0}
\gr_{g_i}(k^{d_i})$$ formed by the pairs $(\eps,E)$
such that $E$ is a submodule of $FY'$ is a constructible
subset. We denote by $\vg$ this constructible subset.
We thus have a constructible function
$\vg\gfl\zb^{2n}$ sending the pair $(\eps,E)$ to
$(\vdim (Fi)^{-1}E,\vdim(Fp)E)$. This function induces
a constructible function $\delta : \wg \gfl \zb^{2n}$,
and the set $\wefg$ is the fiber of $\delta$ above $(e,f)$.
\end{proof}

The fiber above the class $[\eps]$ of the first projection $\wg\fl\pclmy$
is $\{[\eps]\}\times\grg FY'$ and thus all fibers have Euler characteristics
equal to that of $\grg FY$. Therefore we have:
$$
(\ast\ast) \;\;\;\;\; \chi\big{(}\wlmg\big{)} = \chi\big{(}\pclmy\big{)}\chi(\grg FY).
$$

Define $\elef$ to be the variety $\pclm \times \gre FL \times \grf FM$.
Consider the following map:
\begin{eqnarray*}
 \coprod_{Y\in\yc} \wef & \stackrel{\psi}{\gfl} & \elef \\
([\eps],E) & \longmapsto & \big{(}[\eps], (Fp)E, (Fi)^{-1}E\big{)}.
\end{eqnarray*}
By our assumption,
the map $\psi$ is constructible.

Let $\elun$ be the subvariety of $\elef$ formed by the points in the image of $\psi$,
and let $\elde$ be the complement of $\elun$ in $\elef$.

We can now start to compute:
\begin{eqnarray*}
 \dim \cat(L,\susp M)X_L X_M & = & \xs^{-\coi(L\oplus M)}
\sum_{e,f} \chi(L(e,f))\prod_{i=1}^n x_i^{\langle S_i, e+f \rangle_a} \\
 & = & \sum_{e,f} \chi(\elun) \underline{x}^{-\coi (L\oplus M)}
\prod_{i=1}^n x_i^{\langle S_i, e+f \rangle_a} \\
 & & + \sum_{e,f} \chi(\elde) \underline{x}^{-\coi (L\oplus M)}
\prod_{i=1}^n x_i^{\langle S_i, e+f \rangle_a}.
\end{eqnarray*}
Denote by $s_1$ (resp. $s_2$) the first term (resp. second term) in the right hand side
of the last equality above. We first compute the sum $s_1$.

As shown in~\cite{CC}, the fibers of $\psi$ over $\elun$ are affine spaces.
For the convenience of the reader, we sketch a proof.
Let $([\eps],U,V)$ be in $\elun$. Denote by
$Y$ the middle term of $\eps$ and by
$\gruv$ the projection of the fiber $\psi^{-1}([\eps],U,V)$
on the second factor $\gr FY$.
Let $W$ be a cokernel of the injection of $U$ in $FM$.
$$
\xymatrix@R-5pt{
W & & & \\
FM \hae^\pi \dr^i & FY \dr^p & FL \dr & F\susp M \\
U \ham^{i_U} \dr & E \ham \dre & V \ham_{i_V}&
}
$$
\begin{lem}\label{lem: CC}\emph{(Caldero--Chapoton)}
There is a bijection $\homb(V,W) \gfl \gruv $.
\end{lem}

\begin{proof}
Define a free transitive action of $\homb(V,W)$ on $\gruv$
in the following way: For any $E$ in $\gruv$ and any $g$ in
$\homb(V,W)$, define $E_g$ to be the submodule of $FY$ of
elements of the form $i(m)+x$ where $m$ belongs to $FM$,
$x$ belongs to $E$ and $gpx=\pi m$.
Note that $E_g$ belongs to $\gruv$ (since the kernel of $i$
is included in $U$), that $E_0=E$
and that $(E_g)_h=E_{g+h}$.
This action is free: An element $i(m)+x$ is in $E$ if and only if
$m$ is in $U$. This is equivalent to the vanishing of
$\pi m$, which in turn is equivalent to $px$ belonging to
the kernel of $g$.
This action is transitive: Let $E$ and $E'$ be in $\gruv$.
For any $v$ in $V$, let $g(v)$ be $\pi(x'-x)$ where
$x\in E$, $x'\in E'$ and $px=px'=v$. This defines a map
$g:V\gfl W$ such that $E_g=E'$.
\end{proof}

By lemma~\ref{lem: CC}, we obtain the following equality
between the Euler characteristics:
$$
\sum_{\class} \chi(\wef) = \chi(\elun),
$$
which implies the equality
$$
s_1 = \sum_{e,f,\class} \chi\big{(}\wef\big{)}\underline{x}^{-\coi (L\oplus M)}
\prod_{i=1}^n x_i^{\langle S_i, e+f \rangle_a}.
$$

If the pair $([\eps],E)$ belongs to $\wefg$, then by~\cite[lemma 5.1]{Pcc},
we have 
$$
 \sum_{i=1}^n \langle S_i, e+f \rangle_a [P_i] - \coi (L\oplus M)
= \sum_{i=1}^n \langle S_i, g \rangle_a [P_i] - \coi (\mte)
$$
and $\coi(\mte) = \coi Y$ since the morphism $\eps$ is in $\clmy$.
Therefore,
\begin{eqnarray*}
 s_1 & = & \sum_{e,f,g,\class} \chi\big{(}\wefg\big{)}\underline{x}^{-\coi (L\oplus M)}
\prod_{i=1}^n x_i^{\langle S_i, e+f \rangle_a} \\
 & = & \sum_{e,f,g,\class} \chi\big{(}\wefg\big{)}\underline{x}^{-\coi Y}
\prod_{i=1}^n x_i^{\langle S_i, g \rangle_a} \\
 & = & \sum_{g,\class} \chi \big{(} \wlmg \big{)} \underline{x}^{-\coi Y}
\prod_{i=1}^n x_i^{\langle S_i, g \rangle_a} \\
 & = & \sum_{\class} \sum_g \chi\big{(}\pclmy\big{)}\chi(\grg FY) \underline{x}^{-\coi Y}
\prod_{i=1}^n x_i^{\langle S_i, g \rangle_a} \text{ by }(\ast\ast)\\
 & = & \sum_{\class} \chi\big{(}\pclmy\big{)} X_Y.
\end{eqnarray*}

We now consider the sum $s_2$.
Recall that since $\cat$ is $2$-Calabi--Yau, there is an
isomorphism
$$
\phi_{L,M}: \cat(\susm L,M)\gfl
D\cat(M,\susp L).
$$
We denote by $\phi$ the induced
duality pairing:
\begin{eqnarray*}
 \phi : \cat(\susm L,M)\times\cat(M,\susp L) & \gfl & k \\
(a,b) & \longmapsto & \phi_{L,M}(a)b.
\end{eqnarray*}
Let $\cefg$ consist of all pairs $\big{(} ([\eps], U, V) , ([\eta],E) \big{)}$
in $\elde \times \wmlg$ such that
$\phi (\susm\eps,\eta) \neq 0$, $(Fi)^{-1}E = V$ and $(Fp)E = U$,
where $Fi$, $Fp$ are given by $\vph_{M,L}(\eta)$.
The set $\cefg$ is constructible, by our assumption.
Let $\cef$ be the union of all $\cefg$,
where $Y$ runs through the set of representatives
$\yc$, and $g$ through $\ko$.
We then consider the following two projections
$$
\xymatrix{
\cef \bas_{p_1} & \text{and} & \cefg \bas_{p_2} \\
\elde & & \wmlefg .
}
$$

The aim of the next proposition is to show that the projections
$p_1$ and $p_2$ are surjective, and to describe their fibers.

Let $U$ be in $\gre FL$, and $V$ be in $\grf FM$.
Let $U \stackrel{i_U}{\gfl} L$ and $V \stackrel{i_V}{\gfl}M$
lift these two inclusions to the triangulated category $\cat$.
As in section 4 of~\cite{Pcc}, let us consider the
following two morphisms:
$\al$ from
$\cat(\susm L,U)\oplus\cat(\susm L,M)$ to
$\cat/(T)\left( \susm V, U \right) \oplus (\susm V, M )
\oplus  \cat/(\susp T)\left(\susm L, M \right)$
and
$$
 \al':
(\susp T)(U,\susp V)\oplus \cat(M,\susp V) \oplus (\susp^2 T)(M, \susp L)
\gfl
\cat(U, \susp L) \oplus \cat(M,\susp L)
$$
defined by:
$$
\al(a,b) = (a\susm i_V, i_U a \susm i_V - b \susm i_V , i_U a - b )
$$
and
$$
\al'(a,b,c) = \Big{(}(\susp i_V) a + c\,i_U + (\susp i_V)b\,i_U ,
-c - (\susp i_V)b \Big{)}.
$$

Remark that the maps $\al$ and $\al'$ are
dual to each other via the pairing $\phi$.
In the following lemma, orthogonal
means orthogonal with respect to this pairing.

\begin{prop}\emph{\cite[proposition 3]{CK1}}\label{prop: CK1 3}
With the same notations as above,
the following assertions are equivalent:
\begin{itemize}
 \item[(i)] The triple $([\eps], U, V)$ belongs to $\elde$.
 \item[(ii)] The morphism $\susm\eps$ is not orthogonal to $\cat(M,\susp L) \cap \im \al'$.
 \item[(iii)] There is an $\eta \in \cat( M, \susp L )$ such that
$\phi(\susm \eps,\eta) \neq 0$ and such that if
$$L \stackrel{i}{\gfl} N \stackrel{p}{\gfl} M \stackrel{\eta}{\gfl}\susp L$$
is a triangle in $\cat$, then there exists $E\in\gr FN$ with
$(Fi)^{-1}E = V$ and $(Fp)E = U$.
\end{itemize}
\end{prop}

\begin{proof}
Let us start with the equivalence of (i)
and (ii).
The same proof as that in~\cite[proposition 3]{CK1}
applies in this setup:
Denote by $p$ the canonical projection of
$\cat(\susm L,U)\oplus\cat(\susm L,M)$ onto
$\cat(\susm L,M)$. Then, by~\cite[lemma 4.2]{Pcc},
assertion (i) is equivalent to $\susm \eps$ not belonging to
$p(\ke\al)$. That is, the morphism $\susm \eps$
is not in the image of the composition:
$$
q : \ke\al \gfl \cat(\susm L,U)\oplus\cat(\susm L,M)
\gfl \cat(\susm L,M).
$$
So (i) holds if and only if $\susm \eps$
is not in the orthogonal of the orthogonal
of the image of $q$. The orthogonal of the image of $q$
is the kernel of its dual, which is given by
the composition:
$$
\cat(M,\susp L)\gfl \cat(U,\susp L)\oplus\cat(M,\susp L) \gfl
\coke \al'.
$$
Therefore assertion (i) is equivalent to the morphism
$\susm\eps$ not being in the orthogonal of
$\cat(M,\susp L) \cap \im\al'$ which proves that
(i) and (ii) are equivalent.

By~\cite[lemma 4.2]{Pcc}, a morphism in $\cat(M,\susp L)$
is in the image of $\al'$ if and only if it satisfies
the second condition in (iii). Therefore (ii) and (iii) are equivalent.
\end{proof}

A variety $X$ is called an \emph{extension of affine spaces}
in~\cite{CK1} if there is a vector space $V$ and a surjective
morphism $X\gfl V$ whose fibers are affine spaces
of constant dimension. Note that extensions of affine spaces
have Euler characteristics equal to $1$.

\begin{prop}\emph{\cite[proposition 4]{CK1}}\label{prop: CK1 4}
\begin{itemize}
 \item[a)] The projection $\cef \stackrel{p_1}{\gfl} \elde$ is surjective
and its fibers are extensions of affine spaces.
 \item[b)] The projection $\cefg \stackrel{p_2}{\gfl}\wmlefg$ is surjective
and its fibers are affine spaces.
 \item[c)] If $\cefg$ is not empty, then we have
$$\sum_{i=1}^n \langle S_i, e+f \rangle_a [P_i] - \coi (L\oplus M)
=  \sum_{i=1}^n \langle S_i, g \rangle_a [P_i] - \coi Y .$$
\end{itemize}
\end{prop}

\begin{proof}
Let us first prove assertion a). The projection $p_1$
is surjective by the equivalence of i) and iii)
in proposition~\ref{prop: CK1 3}.
Let $X$ be the fiber of $p_1$ above some
$([\eps],U,V)$ in $\elde$. Let $V$ be the set
of all classes $[\eta]$ in $\pb\big(\cat(M,\susp L)\cap\im\al'\big)$
such that $\phi(\susm\eps,\eta)$ does not vanish.
The set $V$ is the projectivization of the complement
in $\cat(M,\susp L)\cap\im\al'$ of the hyperplane
$\ke\phi(\susm\eps,\;)$. Hence $V$ is a vector space.
Let us consider the projection $\pi : X\gfl V$. This projection
is surjective by~\cite[lemma 4.2]{Pcc}. Let $\eta$ represent
a class in $V$,
and let $Fi$, $Fp$ be given by $\vph_{M,L}(\eta)$.
Then the fiber of $\pi$ above $[\eta]$ is given by the
submodules $E$ of $FY$ such that $(Fi)^{-1}E=V$ and $(Fp)E=U$.
Lemma~\ref{lem: CC} thus shows that the fibers of $\pi$
are affine spaces of constant dimension.

Let us prove assertion b). Let $([\eta],E)$ be in $\wmlefg$.
The fiber of $p_2$ above $([\eta],E)$ consists of the elements
of the form $\big(([\eps],U,V),([\eta],E)\big)$ where $U$ and $V$
are fixed submodules given by $[\eta]$ and $E$, and $[\eps]\in\pclm$
is such that $\phi(\susm \eps,\eta)$ does not vanish.
Therefore the projection $p_2$ is surjective and its fibers
are affine spaces.

To prove assertion c), apply~\cite[lemma 5.1]{Pcc} and remark that
if $Y'$ belongs to $\class$, then $Y'$ and $Y$ have the same
coindex.
\end{proof}

As a consequence, we obtain the following equalities:
$$
\chi(\cef) = \chi(\elde) \text{ and } 
\chi\big{(}\cefg\big{)} = \chi\big{(}\wmlefg\big{)}.
$$

We are now able to compute $s_2$ : 
\begin{eqnarray*}
 s_2 & = & \sum_{e,f} \chi(\elde)\underline{x}^{-\coi (L\oplus M)} 
\prod_{i=1}^n x_i^{\langle S_i, e+f \rangle_a} \\
 & = & \sum_{e,f} \chi(\cef)\underline{x}^{-\coi (L\oplus M)} 
\prod_{i=1}^n x_i^{\langle S_i, e+f \rangle_a} \text{ by }\ref{prop: CK1 4}\text{ a)} \\
 & = & \sum_{e,f,g,\class} \chi\big{(}\cefg\big{)}
\underline{x}^{-\coi (L\oplus M)} 
\prod_{i=1}^n x_i^{\langle S_i, e+f \rangle_a} \\
 & = & \sum_{e,f,g,\class} \chi\big{(}\cefg\big{)}
\underline{x}^{-\coi Y} 
\prod_{i=1}^n x_i^{\langle S_i, g \rangle_a} \text{ by }\ref{prop: CK1 4}\text{ c)}\\
 & = & \sum_{e,f,g,\class} \chi\big{(}\wmlefg\big{)}
\underline{x}^{-\coi Y} 
\prod_{i=1}^n x_i^{\langle S_i, g \rangle_a} \text{ by }\ref{prop: CK1 4}\text{ b)}\\
 & = & \sum_{g,\class} \chi\big{(}\wmlg\big{)}\underline{x}^{-\coi Y} 
\prod_{i=1}^n x_i^{\langle S_i, g \rangle_a}\\
 & = & \sum_{g,\class} \chi\big{(}\pcmly\big{)}\chi(\grg FY)
\underline{x}^{-\coi Y} \prod_{i=1}^n x_i^{\langle S_i, g \rangle_a}
\text{ by }(\ast\ast)\\
 & = & \sum_{\class} \chi\big{(}\pcmly\big{)} X_Y.
\end{eqnarray*}

Gathering our results we have:
\begin{eqnarray*}
 \dim \cat(L,\susp M)X_L X_M & = & s_1 + s_2 \\
 & = & \sum_{\class} \chi\big{(}\pclmy\big{)} X_Y + \sum_{\class} \chi\big{(}\pcmly\big{)} X_Y,
\end{eqnarray*}
which proves Theorem~\ref{theo: mf}.

\section{Fu--Keller's cluster character}

In this section, it is proven that the cluster character $X'$ defined
by C. Fu and B. Keller, in \cite{FK}, satisfies a multiplication formula similar
to that of \cite{GLSmf}. In the case of the categories
$\cat_w$ of Geiss--Leclerc--Schr\"oer, when $T$ a reachable cluster-tilting object,
this also follows from~\cite[Theorem 4]{GLSAnsatz}.
Note that the notations used in this section differ from those of \cite{FK}.

Let $k$ be the field of complex numbers and
let $\ec$ be a $k$-linear Hom-finite Frobenius category with split idempotents.
Assume that its stable category $\cat$ is 2-Calabi--Yau and that $\ec$
contains a cluster tilting object $T$. Denote the endomorphism algebra
$\End_\ec(T)$ by $A$ and recall that the algebra $\End_\cat(T)$
is denoted by $B$. The object $T$ is assumed to be basic
with indecomposable summands $T_1,\ldots,T_n$ where
the projective-injective ones are precisely $T_{r+1},\ldots,T_n$.
For $i=1,\ldots,n$, we denote by $S_i$ the simple top
of the projective $A$-module $\ec(T,T_i)$.
The modules in $\modb$ are identified with the modules
in $\moda$ without composition factors isomorphic to one of the
$S_i$, $r<i\leq n$.

\subsection{Statement of the multiplication formula}

We first recall the definition of $X'$ from \cite{FK}.
For any two finitely generated $A$-modules
$L$ and $M$, put
$$
\ps{L}{M}_3 = \sum_{i=0}^3 (-1)^i \dim_k\ext^i_A(L,M).
$$

\vspace{3mm}
\noindent {\bfseries Proposition} [Fu-Keller]:
\emph{If $L,M\in\modb$ have the same image in
$\kzero(\moda)$, then we have
$$
\ps{L}{Y}_3 = \ps{M}{Y}_3
$$
for all finitely generated $A$-module $Y$.}
\vspace{3mm}

Therefore, the number $\ps{L}{S_i}_3$ only depends
on the dimension vector of $L$, for $L\in\modb$.
Put $\ps{\vdim L}{S_i}_3 = \ps{L}{S_i}_3$,
for $i=1,\ldots,n$.

Recall that $F$ is the functor $\cat(T,?)$ from
the category $\cat$ to $\modb$. Denote by $G$
the functor $F\susp \simeq \ext^1_\ec(T,?)$.
For $a\in\kzero(\proj A)$, the notation
$\xs^a$ stands for the product $\prod x_i^{a_i}$,
where $a_i$ is the multiplicity of $[P_i]$ in $a$.

\vspace{3mm}
\noindent {\bfseries Definition} [Fu-Keller]
\emph{For $M\in\ec$, define the Laurent polynomial
$$
X_M' = \xs^{\ind M}\sum_{e\in \kzero(\moda)}
\chi(\gre(GM))\prod_{i=1}^n x_i^{-\ps{e}{S_i}}.
$$}
\vspace{3mm}

Note that the $A$-module $GM$ does not have composition
factors isomorphic to one of the $S_i$, $r<i\leq n$,
so that the sum might as well be taken over
the Grothendieck group $\ko$.

\vspace{3mm}
\noindent {\bfseries Theorem} [Fu-Keller]
\emph{The map $M\mapsto X_M'$ is a cluster character
on $\ec$. It sends $T_i$ to $x_i$, for all $i=1,\ldots,n$}
\vspace{3mm}

For a class $\eps\in\ext^1_\ec(L,M)$,
let $\mte$ be the (isomorphism class of any)
middle term of a conflation
which represents $\eps$.

\begin{theo}\label{theo: FKmf}
For all $L,M\in\ec$, we have
$$
\chi(\pelm)X_L'X_M' =
\int_{[\eps]\in\pelm} X_{\mte}' +
\int_{[\eps]\in\peml} X_{\mte}'.
$$
\end{theo}

The proof of this theorem is postponed to
section \ref{ssection: proof FKmf}. The next
section is dedicated to proving that the map
sending $[\eps]$ to $X_{\mte}'$ is "integrable
with respect to $\chi$".

\subsection{Constructibility}

Let $L,M$ be objects in $\ec$. 
To any $\eps\in\cat(L,\susp M)$ corresponds
a class in $\ext^1_\ec(L,M)$.
Given a conflation in this class,
we denote by $\mte$
the isomorphism class in $\ec$ of
its middle term. Note
that this notation does not coincide
with the one in section \ref{section: notations}.
Nevertheless, those two definitions
yield objects which are isomorphic in $\cat$.

\begin{lem} The map
$$
\lambda : \clm \gfl \kzero(\operatorname{proj} A)
$$
which sends $\eps$ to the index of $\mte$
is constructible.
\end{lem}

\begin{proof}
Let $\eps\in\clm$, and let
$\xymatrix{M\;\dri^i & Y\dre^p & L}$
be a conflation whose class in $\ext^1_\ec(L,M)$
corresponds to $\eps$. Let $d$ be the dimension
vector of $\coke \ec(T,p)$. By the proof of
\cite[lemma 3.4]{FK}, we have
$$
\ind Y = \ind(L\oplus M) - \sum_{i=1}^n \ps{d}{S_i}_3[P_i].
$$
Thanks to lemma \ref{lem: mu constr} and lemma \ref{lem: coker},
the formula above shows that the map $\lambda$ is constructible.
\end{proof}

If an object $Y$ belongs to $\mte$ for some $\eps\in\ext^1_\ec(L,M)$,
we let $\dclass$ denote the set of all isomorphism classes of
objects $Y'\in\ec$ such that:
\begin{itemize}
\item $Y'$ is the middle term of some conflation in $\ext^1_\ec(L,M)$,
\item $\ind Y'=\ind Y$ and
\item for all $e$ in $\ko$, we have
$\chi(\gre GY') = \chi(\gre GY).$
\end{itemize}
We denote the set of all $\eps\in\clm$ with
$\mte\in\dclass$ by $\elmy$.
Note that the set $\elmy$ is a (constructible) subset
of $\clmy$.

As for proposition \ref{prop: clmy constr},
the following proposition follows easily from
corollary \ref{cor: chi constr} and the previous lemma.

\begin{prop}
The sets $\elmy$ are constructible subsets
of $\clm$. Moreover, the set $\clm$
is a finite disjoint union of such constructible subsets.
\end{prop}

This proposition shows that the right hand side
in the multiplication formula of theorem \ref{theo: FKmf}
is well-defined.

\subsection{Proof of theorem \ref{theo: FKmf}}\label{ssection: proof FKmf}

The proof is essentially the same as that in section
\ref{section: proof mf}, where the use of
\cite[lemma 5.1]{Pcc} is replaced by that of
\cite[lemma 3.4]{FK}.

Let $L$ and $M$ be two objects in $\ec$.
Let $\eps$ be a morphism in $\elmy$ for some $Y\in\ec$,
and let
$M\stackrel{-i}{\gfl} Y' \stackrel{-p}{\gfl} L \stackrel{-\eps}{\gfl} \susp M$
be a triangle in $\cat$.
Remark that, by section \ref{ssection: H stable}, the category
$\cat$ has constructible cones.
The image of $\susp\eps$ under
$\vph_{\susp L,\susp M}$ lifts the orbit of the matrix representation
of $\overrightarrow{A_4}$ in $\modb$ given by
$\xymatrix@-.1pc{GM \dr^{Gi} & GY' \dr^{Gp} &
GL \dr^{G\eps\;\;\;} & G\susp M}.$
In all of this section, we will take the liberty of
denoting by $Gi$, $Gp$ and $GY'$ the image $\vph_{\susp L,\susp M}(\susp\eps)$.
Denote by $\Delta$ the dimension vector $\vdim GL + \vdim GM$.
For any non-negative $e$, $f$ and $g$ in $\ko$,
let $\ewefg$ be the subset of
$$\pelmy\times\coprod_{d\leq \Delta}
\prod_{i=1}^n \gr_{g_i}(k^{d_i})$$
formed by the pairs $([\eps],E)$ such that
$E$ is a submodule of $GY'$ of dimension vector $g$,
$\vdim (Gp)E = e$ and $\vdim (Gi)^{-1}E = f$.
We let
\begin{itemize}
 \item $\ewg$ denote the union of all $\ewefg$ with
$e\leq\vdim GL$ and

$f\leq\vdim GM$ and
 \item $\ewef$ denote the union
of all $\ewefg$ with $g\leq \vdim GL + \vdim GM$.
\end{itemize}
Note that, by lemma \ref{lem: wefg constr},
the sets $\ewefg$ are constructible.

The fiber above the class $[\eps]$ of the first projection $\ewg\fl\pelmy$
is $\{[\eps]\}\times\grg GY'$ and thus all fibers have Euler characteristics
equal to that of $\grg GY$. Therefore we have:
$$
(\ast\ast) \;\;\;\;\; \chi\big{(}\ewlmg\big{)} = \chi\big{(}\pelmy\big{)}\chi(\grg GY).
$$

Define $\eelef$ to be the variety $\pclm \times \gre GL \times \grf GM$.
Consider the following map:
\begin{eqnarray*}
 \coprod_{\dclass} \ewef & \stackrel{\psi}{\gfl} & \eelef \\
([\eps],E) & \longmapsto & \big{(}[\eps], (Gp)E, (Gi)^{-1}E\big{)}.
\end{eqnarray*}
By lemma \ref{lem: mu constr},
the map $\psi$ is constructible.

As in section \ref{section: proof mf},
let $\eelun$ be the subvariety of $\eelef$
formed by the points in the image of $\psi$,
and let $\eelde$ be the complement of $\eelun$ in $\eelef$.

Using the notations above, we have
$$
\pclm X_L' X_M' = \xs^{\ind(L\oplus M)}
\sum_{e,f} \chi(\eelef)\prod_{i=1}^n x_i^{-\langle e+f , S_i \rangle_3}
$$
For $j=1,2$, denote by $\sigma_j$ the term
$$
\xs^{\ind(L\oplus M)}
\sum_{e,f} \chi(\lc_j(e,f))\prod_{i=1}^n x_i^{-\langle e+f , S_i \rangle_3}
$$
so that
$\pclm X_L' X_M' = \sigma_1 + \sigma_2$.

We first compute the sum $\sigma_1$.
As shown in~\cite{CC}, the fibers of $\psi$ over $\eelun$ are affine spaces
(see lemma~\ref{lem: CC}). Therefore we have the following equality
between Euler characteristics:
$$
\sum_{\dclass} \chi(\ewef) = \chi(\eelun),
$$
which implies the equality
$$
\sigma_1 = \sum_{e,f,\dclass} \chi\big{(}\ewef\big{)}\xs^{\ind (L\oplus M)}
\prod_{i=1}^n x_i^{-\langle e+f , S_i \rangle_3}.
$$

If the pair $([\eps],E)$ belongs to $\ewefg$, then by~\cite[lemma 3.4]{FK},
we have 
$$
\ind(L\oplus M) - \sum_{i=1}^n \langle e+f , S_i \rangle_3 [P_i]
= \ind (\mte) - \sum_{i=1}^n \langle g , S_i \rangle_3 [P_i]
$$
and $\ind(\mte) = \ind Y$ since the morphism $\eps$ is in $\elmy$.
Therefore,
\begin{eqnarray*}
 \sigma_1 & = & \sum_{e,f,g,\dclass} \chi\big{(}\ewefg\big{)}\xs^{\ind (L\oplus M)}
\prod_{i=1}^n x_i^{-\langle e+f , S_i \rangle_3} \\
 & = & \sum_{e,f,g,\dclass} \chi\big{(}\ewefg\big{)}\xs^{\ind Y}
\prod_{i=1}^n x_i^{-\langle g , S_i \rangle_3} \\
 & = & \sum_{g,\dclass} \chi \big{(} \ewlmg \big{)} \xs^{\ind Y}
\prod_{i=1}^n x_i^{-\langle g , S_i \rangle_3} \\
 & = & \sum_{\dclass} \sum_g \chi\big{(}\pelmy\big{)}\chi(\grg GY) \xs^{\ind Y}
\prod_{i=1}^n x_i^{-\langle g , S_i \rangle_3} \text{ by }(\ast\ast)\\
 & = & \sum_{\dclass} \chi\big{(}\pelmy\big{)} X_Y'.
\end{eqnarray*}

We now consider the sum $\sigma_2$.
Recall that we denote by $\phi$ the
duality pairing:
\begin{eqnarray*}
 \phi : \cat(\susm L,M)\times\cat(M,\susp L) & \gfl & k \\
(a,b) & \longmapsto & \phi_{L,M}(a)b
\end{eqnarray*}
induced by the 2-Calabi--Yau property of $\cat$.

Let $\ecefg$ consist of all pairs $\big{(} ([\eps], U, V) , ([\eta],E) \big{)}$
in $\eelde \times \ewmlg$ such that
$\phi (\susm\eps,\eta) \neq 0$, $(G\iota)^{-1}E = V$ and $(G\pi)E = U$
(where $G\iota$, $G\pi$ are given by $\vph_{\susp M,\susp L}(\susp\eta)$).
The set $\ecefg$ is constructible, by lemma \ref{lem: mu constr}.
Let $\ecef$ be the union of all $\ecefg$,
where $Y$ runs through a set of representatives
for the classes $\dclass$, and $g$ through $\ko$.
We then consider the following two projections
$$
\xymatrix{
\ecef \bas_{\rho_1} & \text{and} & \ecefg \bas_{\rho_2} \\
\eelde & & \ewmlefg .
}
$$

As shown in section \ref{section: proof mf}, the projections
$\rho_1$ and $\rho_2$ are surjective. Moreover, by
\cite[lemma 3.4]{FK},
if $\ecefg$ is not empty, then we have
$$\ind (L\oplus M) - \sum_{i=1}^n \langle e+f ,S_i \rangle_3 [P_i]
= \ind Y - \sum_{i=1}^n \langle g , S_i \rangle_3 [P_i] .$$

As a consequence, we obtain the following equalities:
$$
\chi(\ecef) = \chi(\eelde) \text{ and } 
\chi\big{(}\ecefg\big{)} = \chi\big{(}\ewmlefg\big{)}.
$$

We are now able to compute $\sigma_2$ : 
\begin{eqnarray*}
 \sigma_2 & = & \sum_{e,f} \chi(\eelde)\xs^{\ind (L\oplus M)} 
\prod_{i=1}^n x_i^{-\langle e+f , S_i \rangle_3} \\
 & = & \sum_{e,f} \chi(\ecef)\xs^{\ind (L\oplus M)} 
\prod_{i=1}^n x_i^{-\langle e+f  , S_i \rangle_3} \\
 & = & \sum_{e,f,g,\dclass} \chi\big{(}\ecefg\big{)} \xs^{\ind(L\oplus M)} 
\prod_{i=1}^n x_i^{-\langle e+f , S_i \rangle_3} \\
 & = & \sum_{e,f,g,\dclass} \chi\big{(}\ecefg\big{)} \xs^{\ind Y} 
\prod_{i=1}^n x_i^{-\langle g , S_i \rangle_3} \\
 & = & \sum_{e,f,g,\dclass} \chi\big{(}\ewmlefg\big{)} \xs^{\ind Y} 
\prod_{i=1}^n x_i^{-\langle g , S_i \rangle_3} \\
 & = & \sum_{g,\dclass} \chi\big{(}\ewmlg\big{)} \xs^{\ind Y} 
\prod_{i=1}^n x_i^{-\langle g , S_i \rangle_3} \\
 & = & \sum_{g,\dclass} \chi\big{(}\pemly\big{)}\chi(\grg GY) \xs^{\ind Y}
\prod_{i=1}^n x_i^{-\langle g , S_i \rangle_3} \\
 & = & \sum_{\dclass} \chi\big{(}\pemly\big{)} X_Y'.
\end{eqnarray*}
We thus have:
\begin{eqnarray*}
 \pclm X_L' X_M' & = & \sigma_1 + \sigma_2 \\
 & = & \!\!\!\!\! \sum_{\dclass} \chi\big{(}\pelmy\big{)} X_Y' +
 \sum_{\dclass} \chi\big{(}\pemly\big{)} X_Y',
\end{eqnarray*}
which concludes the proof.

\begin{acknowledgements}\label{ackref}
This article is part of my PhD thesis under the supervision of Bernhard Keller.
I would like to thank him deeply for his outstanding supervision
and for pointing out reference~\cite{JoyceConstr} to me.
I would like to thank Bernt Tore Jensen for his clear explanations on
constructible sets and locally closed sets. The definition of a constructible
map has been changed according to his remarks. Finally, I would like to thank
the anonymous referee for his valuable comments which helped improve the exposition
of the paper.
\end{acknowledgements}

%
%

\affiliationone{
   Yann Palu\\
   School of Mathematics
   University of Leeds\\
   Leeds
   LS2 9JT\\
   U.K.
   \email{ypalu@maths.leeds.ac.uk}}
\end{document}